\documentclass[3p]{elsarticle}

\usepackage{lineno,hyperref}
\usepackage{xie_macros}
\modulolinenumbers[5]

\journal{Linear Algebra and its Applications}

\begin{document}

\begin{frontmatter}

\title{On $2\times 2$ Tropical Commuting Matrices}

\author{Yangxinyu Xie}
\address{Department of Mathematics, University of Texas at Austin, Speedway 2515 Stop C1200, Austin, TX, 78712, USA. }
\ead{yx4247@utexas.edu}

\begin{abstract}
This paper investigates the geometric properties of a special case of the two-sided system given by $2 \times 2$ tropical commuting constraints. Given a finite matrix $A \in \R^{2\times 2}$, the paper studies the extremals of the tropical polyhedral cone generated by the entries of matrices $B$ such that $A \otimes B = B \otimes A$ and proposes a criterion to test whether two $2\times 2$ matrices commute in max linear algebra.
\end{abstract}


\end{frontmatter}


\section{Introduction}

In the max-linear system, we define the \textit{tropical semiring} $(\Rbar, \oplus, \otimes)$ by $\Rbar := \R \cup \{\ninf\}, a \oplus b := \max(a,b), a \otimes b := a + b$ with the additive identity $\ninf$ and the multiplicative identity $0$. The analogue of classical linear algebra in tropical setting is readily extended using the max-plus operations. That is, given matrices $A = (a_{ij}), B = (b_{ij}) \in \Rbar^{n\times n}, i,j \in \n$, we have $(A \otimes B)_{ij} = \bigoplus_{k}(A_{ik} \otimes B_{kj}) = \max_k(A_{ik} + B_{kj})$. Tropical linear algebra has been studied for a wide range of applications, such as scheduling problems \cite{butkovic_2010}, discrete event systems \cite{baccelli92}, control theory \cite{COHEN1999}, statistical inference \cite{PS04} and pairwise ranking \cite{Tran11}. Researchers have been investigating the properties of tropical commuting matrices from different approaches. Algebraically, it has been shown that any two commuting matrices have a common eigenvector \cite{butkovic_2010, SLS12}.
Earlier work has also unfolded some of the tropical analogues of the classical commuting matrices, including the Frobenius normal forms \cite{KSS10}, rank functions and subgroups \cite{IJK12}. However, the question of when two matrices commute remains a mystery. No general algebraic or geometric characterisation of the two matrices has been discovered or proven. Investigations in special subsets of commuting matrices can be found in literature. \cite{LP12} manifests that the space spanned by all matrices commuting with a given normal matrix $A$ is a finite union of alcoved polytopes \cite{Lam2007} and \cite{MT15} shows that for two Kleene Stars $A$ and $B$, if $A \oplus B$ is also a Kleene star, then $A$ and $B$ commute.

To unravel the simplest situation where we are given a matrix $A = (a_{i,j}) \in \R^{2\times 2}, i,j \in \{1, 2\}$ with finite entries $(a_{ij} > \ninf)$, we observe that for a matrix $B \in (b_{i,j}) \Rbar^{2\times 2},  i,j \in \{1, 2\}$ to commute with $A$, it must satisfy the following set of equations
\begin{equation}
    \label{eq: two sided complete system}
	\begin{aligned}
		(A\otimes B)_{11} = (a_{11} \otimes b_{11})\oplus(a_{12} \otimes b_{21}) &= (a_{11} \otimes b_{11}) \oplus (a_{21} \otimes b_{12}) = (B\otimes A)_{11}\\
		(A\otimes B)_{12} = (a_{11} \otimes b_{12})\oplus(a_{12} \otimes b_{22}) &= (a_{12} \otimes b_{11}) \oplus (a_{22} \otimes b_{12}) = (B\otimes A)_{12}\\
		(A\otimes B)_{21} = (a_{21} \otimes b_{11})\oplus(a_{22} \otimes b_{21}) &= (a_{11} \otimes b_{21})\oplus (a_{21} \otimes b_{22}) = (B\otimes A)_{21}\\
		(A\otimes B)_{22} = (a_{21} \otimes b_{12})\oplus(a_{22} \otimes b_{22}) &= (a_{12} \otimes b_{21})\oplus(a_{22} \otimes b_{22}) = (B\otimes A)_{22}
	\end{aligned}
\end{equation}
Equivalently, we can write this set of equations as a tropical two sided system
\begin{equation}
    \label{eq: two sided system}
    C\otimes x = D\otimes x
\end{equation}
where $x := (b_{11}, b_{12}, b_{21}, b_{22})$ and 
\begin{equation*}
	C = 
		\begin{bmatrix}
		a_{11} & \ninf & a_{12} & \ninf\\
		\ninf & a_{11} & \ninf & a_{12}\\
		a_{21} & \ninf & a_{22} & \ninf\\
		\ninf & a_{21} & \ninf & a_{22}\\
	\end{bmatrix}
	 \quad 
	 D =
	 \begin{bmatrix}
		a_{11} & a_{21} & \ninf & \ninf\\
		a_{12} & a_{22} & \ninf & \ninf\\
		\ninf & \ninf & a_{11} & a_{21}\\
		\ninf & \ninf & a_{12} & a_{22}\\
	\end{bmatrix}
\end{equation*}

It has been observed that the solution set of a tropical two sided system can be finitely generated \cite{butkovic1984elimination, gaubert1992theorie, gaubert1997methods}. Several algorithms have been developed to give explicit descriptions of the solution set of a tropical two sided system \cite{butkovic1984elimination, lorenzo2011algorithm, gaubert2012tropical, AGG13}. However, none of the existing algorithms runs in polynomial time. Some partial solutions of the tropical two sided systems are also discussed in \cite{butkovic_2010,SERGEEV20111758, JONES19}.

Let $K$ denote the solution set of the system (\ref{eq: two sided system}),
\begin{equation*}
    K := \{x: x \in \Rbar^4, C\otimes x = D\otimes x \}.
\end{equation*}
In this paper, we observe that $K$ is a finitely generated tropical polyhedral cone. In particular, we characterise all matrices $B \in \Rbar^{2\times 2}$ that commute with a given finite matrix $A \in \R^{2\times 2}$ by describing the basis of $K$.

\begin{thm}
\label{thm: main}
	Let $K$ be defined above. Then the basis of $K$ consists of at most 6 vectors.
\end{thm}

We prove Theorem \ref{thm: main} by showing the following three cases:

\begin{lem}
    \label{lem: a11 > a22}
    Let $A = (a_{i,j}) \in \R^{2\times 2}, i,j \in \{1,2\}$ be a finite matrix with $a_{11} > a_{22}$. Then the set $\{\beta_1, \beta_2, \beta_3, \beta_4\}$ with
    \begin{equation*}
        \beta_1 = \begin{bmatrix}
		    0 \\
		    \ninf\\
		    \ninf\\
		    0
	    \end{bmatrix}
    	\quad
    	\beta_2 = \begin{bmatrix}
		    0 \\
		    \alpha_1\\
		    \ninf\\
		    0
	    \end{bmatrix}
    	\quad
    	\beta_3 = \begin{bmatrix}
		    0 \\
		    \ninf\\
		    \alpha_2\\
		    0
	    \end{bmatrix}
    	\quad
    	\beta_4 = \begin{bmatrix}
		    a_{11}\\
		    a_{12}\\
		    a_{21}\\
		    \ninf
	    \end{bmatrix},
    \end{equation*}
    where $\alpha_1 = \min(a_{12} - a_{11}, a_{22} - a_{21}), \alpha_2 = \min(a_{21} - a_{11}, a_{22} - a_{12})$, forms a basis of $K$.
\end{lem}
In other words, given that $a_{11} > a_{22}$, a matrix $B$ that commutes with $A$ takes the form
\begin{equation}
    B = \Big(\lambda_1 \otimes\begin{bmatrix}
		    0 & \ninf\\
		    \ninf & 0\\
	    \end{bmatrix}\Big) \oplus 
	    \Big(\lambda_2 \otimes\begin{bmatrix}
		    0 & \alpha_1\\
		    \ninf & 0\\
	    \end{bmatrix}\Big) \oplus 
	    \Big(\lambda_3 \otimes\begin{bmatrix}
		    0 & \ninf\\
		    \alpha_2 & 0\\
	    \end{bmatrix}\Big) \oplus 
	    \Big(\lambda_4 \otimes\begin{bmatrix}
		    a_{11} & a_{12}\\
		    a_{21} & \ninf \\
	    \end{bmatrix}\Big)
\end{equation}
where $\lambda_i \in \Rbar, i = 1,2,3,4$. By symmetry, as for $a_{11} < a_{22}$, we have the following lemma. 
\begin{lem}
    \label{lem: a11 < a22}
    Let $A = (a_{i,j}) \in \R^{2\times 2}, i,j \in \{1,2\}$ be a finite matrix with $a_{11} < a_{22}$. Then the set $\{\beta_1, \beta_2, \beta_3, \beta_4\}$ with
    \begin{equation*}
        \beta_1 = \begin{bmatrix}
		    0 \\
		    \ninf\\
		    \ninf\\
		    0
	    \end{bmatrix}
    	\quad
    	\beta_2 = \begin{bmatrix}
		    0 \\
		    \alpha_1\\
		    \ninf\\
		    0
	    \end{bmatrix}
    	\quad
    	\beta_3 = \begin{bmatrix}
		    0 \\
		    \ninf\\
		    \alpha_2\\
		    0
	    \end{bmatrix}
    	\quad
    	\beta_4 = \begin{bmatrix}
		    \ninf \\
		    a_{12}\\
		    a_{21}\\
		    a_{22}
	    \end{bmatrix},
    \end{equation*}
    where $\alpha_1 = \min(a_{12} - a_{22}, a_{11} - a_{21}), \alpha_2 = \min(a_{21} - a_{22}, a_{11} - a_{21})$, forms a basis of $K$.
\end{lem}
The case in which $a_{11} = a_{22}$ is a bit different, especially $\beta_2$ and $\beta_3$. 
\begin{lem}
    \label{lem: a11 = a22}
    Let $A = (a_{i,j}) \in \R^{2\times 2}, i,j \in \{1,2\}$ be a finite matrix with $a_{11} = a_{22}$. Then the set $\{\beta_1, \beta_2, \beta_3, \beta_4, \beta_5, \beta_6\}$ with
    \begin{equation*}
        \beta_1 = \begin{bmatrix}
		    0 \\
		    \ninf\\
		    \ninf\\
		    0
	    \end{bmatrix}
    	\quad
    	\beta_2 = \begin{bmatrix}
		    a_{21} \\
		    a_{11}\\
		    \ninf\\
		    a_{21}
	    \end{bmatrix}
    	\quad
    	\beta_3 = \begin{bmatrix}
		    a_{12} \\
		    \ninf\\
		    a_{11}\\
		    a_{12}
	    \end{bmatrix} \\
    	\beta_4 = \begin{bmatrix}
		    a_{11} \\
		    a_{12}\\
		    a_{21}\\
		    \ninf
	    \end{bmatrix}
	    \quad
    	\beta_5 = \begin{bmatrix}
		    \ninf \\
		    a_{12}\\
		    a_{21}\\
		    a_{22}
	    \end{bmatrix}
	    \quad
    	\beta_6 = \begin{bmatrix}
		    \ninf \\
		    a_{12}\\
		    a_{21}\\
		    \ninf
	    \end{bmatrix}
    \end{equation*}
    forms a basis of $K$.
\end{lem}
In section \ref{background} we will review the definition of the \textit{tropical polyhedral cone}, as well as the properties of its scaled extremals. In section \ref{sec: main proof} we discuss the proof of of Lemma \ref{lem: a11 > a22} and Lemma \ref{lem: a11 = a22}.  Lastly, we study one approach to visualise the polyhedral cone defined in \ref{eq: two sided system} using the \textit{Barycentric coordinates} in section \ref{geo}.
\section{Background in Tropical Polyhedral Cones}
\label{background}
We denote the set of real numbers by $\R$ and define $\Rbar := \R \cup \{- \infty\}$. The \textit{tropical algebra}, typically the \textit{max-linear algebra} defines a \textit{tropical semiring} $(\Rbar, \oplus, \otimes)$ by $a \oplus b := \max(a,b), a \otimes b := a + b$. For any given $a \in \Rbar$, we have that $a \oplus (\ninf) = a$ and $a \otimes 0 = a$. The matrix addition and matrix multiplication are similar to the classical ones. Given matrices $A = (a_{ij}), B = (b_{ij}) \in \Rbar^{n\times n}, i,j \in \n$, we have $(A \oplus B)_{ij} = (A_{ij} \oplus B_{ij}) = \max(A_{ij}, B_{ij})$ and $(A \otimes B)_{ij} = \bigoplus_{k}(A_{ik} \otimes B_{kj}) = \max_k(A_{ik} + B_{kj})$. The identity matrix $I$ is hence $I = (w_{ij})$ such that $w_{ij} = 0$ if $i = j$ and $w_{ij} = \ninf$ otherwise.

Let $K$ be a set of vectors in space $\Rbar^d$. We say that a vector $u \in \Rbar^d$ is a {\it tropical linear combination} of $K$ if $u = \bigoplus_{w \in K}\lambda_w w, \lambda \in \Rbar$ where only finitely number of $\lambda_w > \ninf$. We use $\spn(K)$ to denote the set of all tropical linear combinations of $K$.
\begin{defn}
    A set $K \subseteq \Rbar^d$ is said to be a \textit{tropical polyhedral cone} if $K = \spn(K)$. If there is a subset $S \subseteq K$ such that $\spn(S) = K$, we call $S$ a set of \textit{generators} for $K$.
\end{defn}
For the purpose of our discussion, we restrict our attention to finitely generated tropical polyhedral cones. That is, there is a set $S$ of finite order and $\spn(S) = K$. A \textit{tropical linear halfspace} is the set of vectors $x$ satisfying $\mathbf{c}^T \otimes x \le \mathbf{d}^T\otimes x$ where $\mathbf{c}, \mathbf{d} \in \Rbar^{d}$. The following theorem gives an equivalent definition of finitely generated tropical polyhedral cones. We omit the proof here.

\begin{thm}[\cite{gaubert1992theorie, gaubert1997methods}; see \cite{gaubert2011minimal}, Theorem 1]
    A tropical polyhedral cone is finitely generated if and only if it is the intersection of finitely many half-spaces.
\end{thm}

In other words, a finitely generated tropical polyhedral cone is the set of vectors $x$ satisfying $C \otimes x \le D\otimes x$ where $C, D \in \Rbar^{m\times d}$ for some integer $m$. For two matrices
\begin{equation}
	A =  \begin{bmatrix}
		a_{11} & a_{12}\\
		a_{21} & a_{22}\\
	\end{bmatrix}, a_{i,j} \in \R, i,j \in \{1,2\}, \quad
	B =  \begin{bmatrix}
		b_{11} & b_{12}\\
		b_{21} & b_{22}\\
	\end{bmatrix}, b_{i,j} \in \Rbar, i,j \in \{1,2\}, 
\end{equation}
If $A$ and $B$ commute, we obtain a two sided system $C\otimes x = D\otimes x, x := (b_{11}, b_{12}, b_{21}, b_{22})$ as (\ref{eq: two sided system}). $C\otimes x = D\otimes x$ gives a special type of tropical polyhedral cone -- it is the intersection of two polyhedral cones $C\otimes x \le D\otimes x$ and $C\otimes x \ge D\otimes x$.

Let $K$ be a tropical polyhedral cone. An element $u \in K$ is called an {\it extremal} in $K$ if for any two elements $v,w \in K$ such that $u = v \oplus w$, we have either $u = v$ or $u = w$. The \textit{support} of a vector $v \in \Rbar^d$, denoted by $\supp(v)$ is the ordered set of indices $j \in \{1,...,d\}$ such that $v_j > \ninf$. Let $S \subseteq \Rbar^d$ be a set of vectors. For a vector $u \in \Rbar^d$ and an index $t \in \supp(u)$ with $u_t > \ninf$, we define
\begin{equation*}
    S_u(t) := \{(u_t - v_t) \otimes v: v \in S, t \in \supp(v)\}
\end{equation*}
That is, for each vector $v \in S$ with support at $t$, we shift it linearly so that its $t$-th entry equals to $u_t$ and put it in $S_u(t)$. An element $v \in S$ is said to be \textit{minimal} if for all $u \in S, u \le v$, we have $u = v$. We stress that for a support $t$, the set $S_v(t)$ may have more than 1 minimal elements. The following proposition allows us to check extremality of a vector by its minimality in the generating set.

\begin{prop}[\cite{BSS07}, Theorem 14; \cite{butkovic_2010}, Proposition 3.3.6]
 	\label{prop: check extremality by minimality}
 	Let $K$ be a tropical polyhedral cone and $S$ be its set of generators. A vector $v$ is an extremal in $K$ if and only if there exists some $t \in \supp(v)$ such that $v$ is a minimal element in the set $S_v(t)$. In other words, for all $w \in S_v(t), w \le v$ implies $w = v$.
\end{prop}

\begin{proof}
        $(\Rightarrow)$ Suppose that $v$ is an extremal in $K$. Then for any $t \in \supp(v)$,
        \begin{equation*}
            v = \bigoplus_{w \in S}(\lambda_w \otimes w)
        \end{equation*}
        for finitely many $\lambda_w > \ninf$. Because $v$ is extremal, we have $v = \lambda_w \otimes w$ for some $w \in S_v(t)$ and $\lambda_w \in \Rbar$. Because $w_t= v_t$, we have $\lambda_w = 0$ and thus $w = v$. If $|\supp(v)| = 1$, then $v$ is minimal in $S_v(t)$ because for any $w \in S_v(t), w \le v$, we have $w_t = v_t$.
        
        Suppose $|\supp(v)| > 1$ and $v$ is not minimal in $S_v(t)$ for all $t \in \supp(v)$. Then for all $t \in \supp(v)$, there exists a vector $w(t) \in S_t(v)$ such that 
        \begin{equation*}
            v = \bigoplus_{t \in \supp(v)}w(t)
        \end{equation*}
        but $v$ equals none of $w(t)$. This contradicts the extremality of $v$.
        
        $(\Leftarrow)$ Let $v$ and $t \in \supp(v)$ be given such that $v$ is minimal in $S_v(t)$. We first show that $v$ is minimal in $K_v(t)$. Assume that there exists a vector $u \in K_v(t)$ such that $u \le v$. Because $u \in K_v(t)$ and $K = \spn(S)$, we have
        \begin{equation}
            \label{eq: some w is smaller than u}
            u = \bigoplus_{w \in S}(\lambda_w \otimes w)
        \end{equation}
        for finitely many $\lambda_w > \ninf$. Notice that there must be a $w \in S_v(t)$ with $w \le u$ and $w_t = u_t$. Because $v$ is minimal in $S_v(t)$, we must have $v \le w$. Therefore, it must be the case that $v = w = u$.
        
        Now, suppose that $v = x \oplus y$ for some $x, y \in K$. Then both $x \le v$ and $y \le v$ and either $x_t = v_t$ or $y_t = v_t$. Without loss of generality, we assume that $x_t = v_t$. Then $x\in K_v(t)$. As $v_t$ is minimal in $K_v(t)$, we must have $x = v$. Hence, $v$ is an extremal in $K$.
\end{proof} 
 
A set $S \subseteq \Rbar^d$ is said to be \textit{independent} if none of the elements in $S$ is a tropical linear combination of other elements in $S$ and is \textit{dependent} otherwise. We say that a vector $v$ is \textit{scaled} if its first finite entry equals $0$. That is, if $i$ is the smallest element in $\supp(v)$, then $v_i = 0$. 

\begin{lem}[\cite{BSS07}, Lemma 7; \cite{butkovic_2010}, Lemma 3.3.1]
    \label{lem: extremals are subset of generators}
    Let $K$ be a tropical polyhedral cone. Let $v$ be a scaled extremal of $K$ and $S$ be a set of scaled generators of $K$. Then $v \in S$.
\end{lem}

\begin{proof}
    Because $v \in K$, 
        \begin{equation*}
            v = \bigoplus_{w \in S} \lambda_w \otimes w
        \end{equation*}
    for finitely many $\lambda_w > \ninf$. Because $v$ is an extremal in $K$, then $v = \lambda_w \otimes w$ for some $w \in S$ and some $\lambda_w > \ninf$. Hence $\supp(w) = \supp(v)$. Because both $v$ and $w$ are scaled, we must have $\lambda_w = 0$ and $v = w$. This implies $v \in S$.
\end{proof}

\begin{lem}[\cite{BSS07}, Lemma 8; \cite{butkovic_2010}, Lemma 3.3.2]
 	\label{lem: independence}
 	Let $K$ be a tropical polyhedral cone. Then any set of scaled extremals of $K$ is independent. 
\end{lem}
\begin{proof}
    Let $E$ be a set of scaled extremals of $K$ and $v \in E$. Let $K':= \spn(E \setminus \{v\}) \subseteq K$. Suppose that $v \in K'$. Then by Lemma \ref{lem: extremals are subset of generators}, $v \in E \setminus \{v\}$ and gives a contradiction.
\end{proof}

For a tropical polyhedral cone $K$, we say that a set $S \subseteq K$ is a \textit{basis} for $K$ if it is an independent set of generators for $K$.

\begin{thm}[\cite{wagneur1991moduloids}, Theorem 5; \cite{BSS07}, Theorem 18]
    \label{thm:all scaled extremals form a basis}
    Let $K$ be a tropical polyhedral cone and $E$ be its the set of all scaled extremals. Then $E$ is a basis for $K$.
\end{thm}
\begin{proof}
    We first show that $E$ generates $K$. Let $S$ be a set of scaled generators of $K$. By Lemma \ref{lem: extremals are subset of generators}, we have $E \subseteq S$. Let $u \in S, u \not \in E$, we show that $S \setminus \{u\}$ is still a set of scaled generators of $K$. By Proposition \ref{prop: check extremality by minimality}, we have that for any $t \in \supp(u)$, $u$ is not minimal in $S_u(t)$. Hence, 
    \begin{equation}
        u = \bigoplus_{t \in \supp(u)} w_t
    \end{equation}
    where $w_t$ is the minimal element in $S_u(t)$, is a linear combination of $S$. Hence, $u \in \spn(S \setminus \{u\})$. Consequently, we have that $E$ generates $K$. By Lemma \ref{lem: independence}, we have that $E$ is independent. Therefore, $E$ forms a basis for $K$. 
\end{proof}

The \textit{Barycentric coordinate} is an alternative for visualisation when infinite entries exist. The classical triangle coordinate $T \in \R^2$ consists of three vertices $v_1, v_2, v_3$. We denote the coordinates of the three vertices as $v_1 = (1, 0, 0), v_2 = (0, 1, 0), v_3 = (0, 0, 1)$ and the area of $T$ as $A_T$. Suppose we have a point $x$ inside $T$ which splits the triangle into three subareas $A_1, A_2, A_3$, as shown in Figure \ref{Bary}. We define the Barycentric coordinate of $x$ as the ratios of the subareas,
\begin{equation}
	\label{BaryCoor}
	x = \phi_1 v_1 + \phi_2 v_2 + \phi_3 v_3 \quad \text{where} \quad \phi_i = \frac{A_i}{A_T}, i \in \{1, 2, 3\}
\end{equation}
For details of the Barycentric coordinates, we refer the reader to \cite{floater15}. 
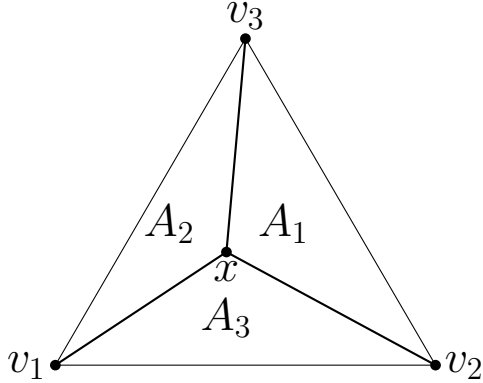
\begin{figure}
\centering
\begin{tikzpicture}[font=\LARGE] 
\def \l{5.00000000000000 } 

\coordinate (A) at (0,0); 
\coordinate (B) at ({\l}, 0);
\coordinate (C) at ({\l / 2 }, {\l * cos(30)});
\coordinate (D) at (\l * 0.45, \l * 0.3);
\coordinate (E) at (\l * 0.6, \l * 0.45);
\coordinate (F) at (\l * 0.3, \l * 0.45);
\coordinate (G) at (\l * 0.45, \l * 0.2);

\fill[black]  (A) circle [radius=2pt]; 
\fill[black]    (B) circle [radius=2pt]; 
\fill[black]    (C) circle [radius=2pt]; 
\fill[black]    (D) circle [radius=2pt];

\draw[black,-] 
                        (B)  --  (A);
\draw[black,-] 
                        (B)  --  (C);
\draw[black,-] 
                        (C)  --  (A);
\draw[black, -, thick] 
                        (B)  --  (D);
\draw[black, -, thick] 
                        (A)  --  (D);
\draw[black, -, thick] 
                        (C)  --  (D);

\draw 
      (A)  node [left]       {$v_1$}
      (B)  node [right]       {$v_2$}
      (C)  node [above]     {$v_3$}
      (D)  node [below]     {$x$}
      (E)  node [below]     {$A_1$}
      (F)  node [below]     {$A_2$}
      (G)  node [below]     {$A_3$};
\end{tikzpicture}
\caption{Barycentric Coordinate}
\label{Bary}
\end{figure}
\begin{notn}
Starting from the section \ref{sec: main proof}, we assume the cone $K$ is defined by the two sided system $C\otimes x = D\otimes x$ in (\ref{eq: two sided system}). We also use $B_{\beta_i}, i \in \{1, 2, 3, 4 ,5 ,6\}$ to denote a $2 \times 2$ matrix $B$ whose entries are 
\begin{equation}
	B =  \begin{bmatrix}
		\beta_{i_1} & \beta_{i_2}\\
		\beta_{i_3} & \beta_{i_4}\\
	\end{bmatrix}
\end{equation}
\end{notn}


\section{Proof of Theorem \ref{thm: main}}
\label{sec: main proof}

In this section we present the proof of Theorem \ref{thm: main} by showing Lemma \ref{lem: a11 > a22} and Lemma \ref{lem: a11 = a22}. The general frame is to first prove that the $\beta$s given in the lemma generates $K$ and then show that each $\beta$ is an extremal in $K$ by verifying its minimality. By Theorem \ref{thm:all scaled extremals form a basis}, we have that these $\beta$s form a basis for $K$. For convenience, we rewrite the two sided system (\ref{eq: two sided complete system}) as the following set of equations. 

\begin{equation}
    \label{eq: two sided complete system rewrite}
    \begin{aligned}
        (A\otimes B)_{11} = \max(a_{11} + b_{11}, a_{12} + b_{21}) &= \max(a_{11} + b_{11}, a_{21} + b_{12}) = (B\otimes A)_{11}\\
		(A\otimes B)_{12} = \max(a_{11} + b_{12}, a_{12} + b_{22}) &= \max(a_{12} + b_{11}, a_{22} + b_{12}) = (B\otimes A)_{12}\\
		(A\otimes B)_{21} = \max(a_{21} + b_{11}, a_{22} + b_{21}) &= \max(a_{11} + b_{21}, a_{21} + b_{22}) = (B\otimes A)_{21}\\
		(A\otimes B)_{22} = \max(a_{21} + b_{12}, a_{22} + b_{22}) &= \max(a_{12} + b_{21}, a_{22} + b_{22}) = (B\otimes A)_{22}
    \end{aligned}
\end{equation}

We assume that a matrix 
\begin{equation}
	B =  \begin{bmatrix}
		\ninf & \ninf\\
		\ninf & \ninf\\
	\end{bmatrix}
\end{equation}
commutes with any matrix $A \in \Rbar^{2\times 2}$ and hence this case will be implicit for the rest of this section. 

\begin{proof}[proof of Lemma \ref{lem: a11 > a22}]
    Assume that $a_{11} > a_{22}$. Then the two sided system (\ref{eq: two sided complete system rewrite}) gets reduced to 
    \begin{equation}
        \label{eq: a11 > a22; first}
        (A\otimes B)_{11} = \max(a_{11} + b_{11}, a_{12} + b_{21}) = \max(a_{11} + b_{11}, a_{21} + b_{12}) = (B\otimes A)_{11}
    \end{equation}
    \begin{equation}
        \label{eq: a11 > a22; second}
        (A\otimes B)_{12} = \max(a_{11} + b_{12}, a_{12} + b_{22}) = a_{12} + b_{11} = (B\otimes A)_{12}
    \end{equation}
    \begin{equation}
        \label{eq: a11 > a22; third}
        (A\otimes B)_{21} = a_{21} + b_{11} = \max(a_{11} + b_{21}, a_{21} + b_{22}) = (B\otimes A)_{21}
    \end{equation}
    \begin{equation}
        \label{eq: a11 > a22; fourth}
        (A\otimes B)_{22} = \max(a_{21} + b_{12}, a_{22} + b_{22}) = \max(a_{12} + b_{21}, a_{22} + b_{22}) = (B\otimes A)_{22}
    \end{equation}
    Let $b = (b_{11}, b_{12}, b_{21}, b_{22}) \in K$. Then it must be the case that $b_{11} > 0$, or $b = (\ninf, \ninf, \ninf, \ninf).$ Let $S = \{\beta_i: i = 1,2,3,4\}$. We first show that $K \subseteq \spn(S)$ by two cases.
    \begin{itemize}
        \item Suppose $a_{21} + b_{12} = a_{12} + b_{21}.$
        In this case we have that both (\ref{eq: a11 > a22; first}) and (\ref{eq: a11 > a22; fourth}) vacuously hold. If $b_{11} = b_{22}$, then (\ref{eq: a11 > a22; second}) and (\ref{eq: a11 > a22; third}) imply
        \begin{equation*}
            b_{12} \le a_{12} - a_{11} + b_{11} \quad \textup{and} \quad b_{21} \le a_{21} - a_{11} + b_{11}
        \end{equation*}
        If $b_{11} \ne b_{22}$, then we must have $b_{11} > b_{22}$; otherwise $a_{12} + b_{22} > a_{12} + b_{11}$, which fails equation (\ref{eq: a11 > a22; second}); moreover, (\ref{eq: a11 > a22; second}) and (\ref{eq: a11 > a22; third}) imply
        \begin{equation*}
            b_{12} = a_{12} - a_{11} + b_{11} \quad \textup{and} \quad b_{21} = a_{21} - a_{11} + b_{11}
        \end{equation*}
        Hence, any $b \in K$ with $a_{21} + b_{12} = a_{12} + b_{21}$ can be written as
        \begin{equation}
            b = (\lambda_1 \otimes \begin{bmatrix}
    		    0 \\
    		    \ninf\\
    		    \ninf\\
    		    0
    	        \end{bmatrix}) \oplus 
    	        (\lambda_2 \otimes \begin{bmatrix}
    		    a_{11} \\
    		    a_{12}\\
    		    a_{21}\\
    		    \ninf
    	        \end{bmatrix})
        \end{equation}
        for some $\lambda_1, \lambda_2 \in \Rbar$.
        \item Suppose $a_{21} + b_{12} \ne a_{12} + b_{21}.$
        Then it must be the case that $b_{22} = b_{11}$. To see this, we first note that to satisfy (\ref{eq: a11 > a22; second}), it cannot be the case that $b_{22} > b_{11}$. If $b_{22} < b_{11}$, then (\ref{eq: a11 > a22; second}) and (\ref{eq: a11 > a22; third}) get reduced to
        \begin{equation*}
            a_{11} + b_{12} = a_{12} + b_{11} \quad \textup{and} \quad a_{21} + b_{11} = a_{11} + b_{21}
        \end{equation*}
        which leads to $a_{21} + b_{12} = a_{12} + b_{21}$. Now, for (\ref{eq: a11 > a22; first}) and (\ref{eq: a11 > a22; fourth}) to hold, because $a_{22} < a_{11},$ we must have
        \begin{equation*}
            b_{12} \le a_{22} - a_{21} + b_{22} = a_{22} - a_{21} + b_{11} \quad \textup{and} \quad b_{21} \le a_{22} - a_{12} + b_{22} = a_{22} - a_{12} + b_{11}
        \end{equation*}
        By (\ref{eq: a11 > a22; second}) and (\ref{eq: a11 > a22; third}), we also require
        \begin{equation*}
            b_{12} \le a_{12} - a_{11} + b_{11} \quad \textup{and} \quad b_{21} \le a_{21} - a_{11} + b_{11}
        \end{equation*}
        Hence, we have that any $b \in K$ with $a_{21} + b_{12} \ne a_{12} + b_{21}$ can be written in the form
        \begin{equation*}
            b = (\lambda_1 \otimes \begin{bmatrix}
    		    0 \\
    		    \ninf\\
    		    \ninf\\
    		    0
    	        \end{bmatrix}) \oplus 
    	        (\lambda_2 \otimes \begin{bmatrix}
    		    0 \\
    		    \alpha_1\\
    		    \ninf\\
    		    0
    	        \end{bmatrix}) \oplus 
    	        (\lambda_3 \otimes \begin{bmatrix}
    		    0 \\
    		    \ninf\\
    		    \alpha_2\\
    		    0
    	        \end{bmatrix})
        \end{equation*}
        for some $\lambda_1, \lambda_2, \lambda_3 \in \Rbar$. Recall that  $\alpha_1 = \min(a_{12} - a_{11}, a_{22} - a_{21}), \alpha_2 = \min(a_{21} - a_{11}, a_{22} - a_{12})$.
    \end{itemize}
    Conversely, for $i \in \{1,2,3\}$,
    \begin{equation*}
            A \otimes B_{\beta_i} = \begin{bmatrix}
        		a_{11} & a_{12}\\
        		a_{21} & a_{22}\\
        	\end{bmatrix}
        	= B_{\beta_i} \otimes A
    \end{equation*}
    and 
    \begin{equation*}
            B_{\beta_4} \otimes A = \begin{bmatrix}
        		(a_{11} \otimes a_{11})\otimes (a_{12} \oplus a_{21}) & a_{11} \otimes a_{12}\\
        		a_{11} \otimes a_{21} &a_{12} \otimes a_{21}\\
        	\end{bmatrix} = B_{\beta_4} \otimes A
    \end{equation*}
    This means all such $b \in \spn(S)$ gives a matrix that commutes with $A$. Hence $\spn(S) = K$ and thus $S$ generates $K$.
    
    It remains to prove that all elements in $S$ are extremals in $K$. By Proposition \ref{prop: check extremality by minimality}, it suffices to show each $\beta_i$ is minimal in $S_{\beta_i}(t)$ for some $t \in \{1,2,3,4\}$. This is obvious because $\beta_4$ is not comparable with any other $\beta$s, while $S_{\beta_2}(2)$ and $S_{\beta_3}(3)$ have only 1 element remaining after removing $\beta_4$; clearly, if we scale $\beta_2$ and $\beta_3$, then $\beta_1$ is smaller than both. Hence, by Lemma \ref{lem: extremals are subset of generators} and Theorem \ref{thm:all scaled extremals form a basis}, $S$ is a basis for $K$.
\end{proof}

Note that the proof of Lemma \ref{lem: a11 < a22} can be proved similarly by symmetry. Now we show the proof of Lemma \ref{lem: a11 = a22}, which is very similar to the proof of Lemma \ref{lem: a11 > a22} we just saw.
\begin{proof}[proof of Lemma \ref{lem: a11 = a22}]
    Assume that $a_{11} = a_{22}$. Then the two sided system (\ref{eq: two sided complete system rewrite}) can be written as 
    \begin{equation}
        \label{eq: a11 = a22; first}
        (A\otimes B)_{11} = \max(a_{11} + b_{11}, a_{12} + b_{21}) = \max(a_{11} + b_{11}, a_{21} + b_{12}) = (B\otimes A)_{11}
    \end{equation}
    \begin{equation}
        \label{eq: a11 = a22; second}
        (A\otimes B)_{12} = \max(a_{11} + b_{12}, a_{12} + b_{22}) = \max(a_{12} + b_{11}, a_{11} + b_{12}) = (B\otimes A)_{12}
    \end{equation}
    \begin{equation}
        \label{eq: a11 = a22; third}
        (A\otimes B)_{21} = \max(a_{21} + b_{11}, a_{11} + b_{21}) = \max(a_{11} + b_{21}, a_{21} + b_{22}) = (B\otimes A)_{21}
    \end{equation}
    \begin{equation}
        \label{eq: a11 = a22; fourth}
        (A\otimes B)_{22} = \max(a_{21} + b_{12}, a_{11} + b_{22}) = \max(a_{12} + b_{21}, a_{11} + b_{22}) = (B\otimes A)_{22}
    \end{equation}
    Let $b = (b_{11}, b_{12}, b_{21}, b_{22}) \in K$ and $S = \{\beta_i: i = 1,2,3,4,5,6\}$. We first show that $K \subseteq \spn(S)$ by two cases.
    \begin{itemize}
        \item Suppose $a_{21} + b_{12} = a_{12} + b_{21}$.
        In this case (\ref{eq: a11 = a22; first}) and (\ref{eq: a11 = a22; fourth}) hold trivially. Notice that $b_{12}$ and $b_{21}$ can both be $\ninf$. If $b_{11} = b_{22}$, even when both of them are $\ninf$, (\ref{eq: a11 = a22; second}) and (\ref{eq: a11 = a22; third}) also hold trivially. 
        Suppose $b_{11} \ne b_{22}$. If $b_{11} > b_{22}$, then (\ref{eq: a11 = a22; second}) and (\ref{eq: a11 = a22; third}) imply
        \begin{equation*}
            b_{22} < b_{11} \le b_{12} - a_{12} + a_{11} \quad \textup{and} \quad b_{22} < b_{11} \le b_{21} - a_{21} + a_{11}.
        \end{equation*}
        On the other hand, if $b_{11} < b_{22}$, then (\ref{eq: a11 = a22; second}) and (\ref{eq: a11 = a22; third}) imply
        \begin{equation*}
            b_{11} < b_{22} \le b_{12} - a_{12} + a_{11} \quad \textup{and} \quad b_{11} < b_{22} \le b_{21} - a_{21} + a_{11}.
        \end{equation*}
        Hence, any $b \in K$ with $a_{21} + b_{12} = a_{12} + b_{21}$ can be written in the form
        \begin{equation}
            b = (\lambda_1 \otimes \begin{bmatrix}
    		    0 \\
    		    \ninf\\
    		    \ninf\\
    		    0
    	        \end{bmatrix}) \oplus 
    	        (\lambda_2 \otimes \begin{bmatrix}
    		    a_{11} \\
    		    a_{12}\\
    		    a_{21}\\
    		    \ninf
    	        \end{bmatrix})\oplus 
    	        (\lambda_3 \otimes \begin{bmatrix}
    		    \ninf \\
    		    a_{12}\\
    		    a_{21}\\
    		    a_{22}
    	        \end{bmatrix})\oplus 
    	        (\lambda_2 \otimes \begin{bmatrix}
    		    \ninf \\
    		    a_{12}\\
    		    a_{21}\\
    		    \ninf
    	        \end{bmatrix})
        \end{equation}
        for some $\lambda_1, \lambda_2, \lambda_3, \lambda_4 \in \Rbar$.
        \item Suppose $a_{21} + b_{12} \ne a_{12} + b_{21}.$
        For (\ref{eq: a11 = a22; first})  to hold, we must have
        \begin{equation*}
            b_{12} \le a_{11} - a_{21} + b_{11} \quad \textup{and} \quad b_{21} \le a_{11} - a_{12} + b_{11}
        \end{equation*}
        and (\ref{eq: a11 = a22; fourth}) requires
        \begin{equation*}
            b_{12} \le a_{11} - a_{21} + b_{22} \quad \textup{and} \quad b_{21} \le a_{11} - a_{12} + b_{22}
        \end{equation*}
        If $b_{11} = b_{22}$, then (\ref{eq: a11 > a22; second}) and (\ref{eq: a11 > a22; third}) hold trivially. Thus any $b \in K$ with $a_{21} + b_{12} \ne a_{12} + b_{21}$ and $b_{11} = b_{22}$ can be written in the form
        \begin{equation}
            b = (\lambda_1 \otimes \begin{bmatrix}
    		    0 \\
    		    \ninf\\
    		    \ninf\\
    		    0
    	        \end{bmatrix}) \oplus 
    	        (\lambda_2 \otimes \begin{bmatrix}
    		    a_{21} \\
    		    a_{11}\\
    		    \ninf\\
    		    a_{21}
    	        \end{bmatrix}) \oplus 
    	        (\lambda_3 \otimes \begin{bmatrix}
    		    a_{12} \\
    		    \ninf\\
    		    a_{11}\\
    		    a_{12}
    	        \end{bmatrix})
        \end{equation}
        for some $\lambda_1, \lambda_2, \lambda_3 \in \Rbar$. 
        Suppose $b_{11} \ne b_{22}$. 
        If $b_{11} > b_{22}$, by (\ref{eq: a11 > a22; second}) and (\ref{eq: a11 > a22; third}), 
        \begin{equation}
            b_{12} \ge a_{12} - a_{11} + b_{11} \quad \textup{and} \quad b_{21} \ge a_{21} - a_{11} + b_{11}
        \end{equation}
        If $b_{11} < b_{22}$, by (\ref{eq: a11 > a22; second}) and (\ref{eq: a11 > a22; third}), 
        \begin{equation}
            b_{12} \ge a_{12} - a_{11} + b_{22} \quad \textup{and} \quad b_{21} \ge a_{21} - a_{11} + b_{22}
        \end{equation}
        We stress that $(a_{11} \otimes a_{11}) \ge (a_{12} \otimes a_{21})$ must hold for $b_{11} \ne b_{22}$ to be possible. Hence, given that $(a_{11} \otimes a_{11}) \ge (a_{12} + a_{21})$, we have that any $b \in K$ with $a_{21} + b_{12} \ne a_{12} + b_{21}$ can be written in the form
        \begin{equation}
            b = (\lambda_1 \otimes \begin{bmatrix}
    		    0 \\
    		    a_{11} - a_{21}\\
    		    \ninf\\
    		    0
    	        \end{bmatrix}) \oplus 
    	        (\lambda_2 \otimes \begin{bmatrix}
    		    0 \\
    		    \ninf\\
    		    a_{11} - a_{12}\\
    		    0
    	        \end{bmatrix})\oplus 
    	        (\lambda_3 \otimes \begin{bmatrix}
    		    0 \\
    		    a_{12} - a_{11}\\
    		    a_{21} - a_{11}\\
    		    \ninf
    	        \end{bmatrix})\oplus 
    	        (\lambda_4 \otimes \begin{bmatrix}
    		    \ninf\\
    		    a_{12} - a_{22}\\
    		    a_{21} - a_{22}\\
    		    0
    	        \end{bmatrix})
        \end{equation}
        for some $\lambda_1, \lambda_2, \lambda_3, \lambda_4 \in \Rbar$. Equivalently, $b$ is a linear combination of $\{\beta_2, \beta_3, \beta_4, \beta_5\}$.
    \end{itemize}
    Conversely,
    \begin{equation*}
        \begin{split}
            A \otimes B_{\beta_1} &= \begin{bmatrix}
        		a_{11} & a_{12}\\
        		a_{21} & a_{22}\\
        	\end{bmatrix}
        	= B_{\beta_1} \otimes A\\
        	A \otimes B_{\beta_2} &= \begin{bmatrix}
        		a_{11} \otimes a_{21} & (a_{11}\otimes a_{11})\oplus(a_{12} \otimes a_{21})\\
        		a_{21} \otimes a_{21} & a_{11} \otimes a_{21}\\
        	\end{bmatrix}
        	= B_{\beta_2} \otimes A\\
        	A \otimes B_{\beta_3} &= \begin{bmatrix}
        		a_{11} \otimes a_{12} & a_{12} \otimes a_{12}\\
        		(a_{11}\otimes a_{11})\oplus(a_{12} \otimes a_{21}) & a_{11} \otimes a_{12}\\
        	\end{bmatrix}
        	= B_{\beta_3} \otimes A\\
            A \otimes B_{\beta_4} &= \begin{bmatrix}
        		(a_{11} \otimes a_{11})\otimes (a_{12} \oplus a_{21}) & a_{11} \otimes a_{12}\\
        		a_{11} \otimes a_{21} &a_{12} \otimes a_{21}\\
        	\end{bmatrix} = B_{\beta_4} \otimes A\\
        	A \otimes B_{\beta_5} &= \begin{bmatrix}
        		a_{12} \otimes a_{21} & a_{11} \otimes a_{12}\\
        		a_{11} \otimes a_{21} &(a_{11} \otimes a_{11})\otimes (a_{12} \oplus a_{21})\\
        	\end{bmatrix} = B_{\beta_5} \otimes A\\
        	A \otimes B_{\beta_6} &= \begin{bmatrix}
        		a_{12} \otimes a_{21} & a_{11} \otimes a_{12}\\
        		a_{11} \otimes a_{21} &a_{12} \otimes a_{21}\\
        	\end{bmatrix} = B_{\beta_6} \otimes A
        \end{split}
    \end{equation*}
    This means all such $b \in \spn(S)$ gives a matrix that commutes with $A$. Hence $\spn(S) = K$ and $S$ generates $K$.
    
    Now we prove that all elements in $S = \{\beta_i, i = 1,2,3,4,5,6\}$ are extremals in $K$ by showing that each $\beta_i$ is minimal in $S_{\beta_i}(t)$ for some $t \in \{1,2,3,4,5,6\}$. Let $S(t) \subseteq S$ be a subset that includes all elements in $S$ with support at $t$ and define $S_0(t)$ by normalising vectors in $S(t)$ such that their $t$-th coordinates are set to zero:
    \begin{equation}
        S_0(t) := \{(- v_t)\otimes v: v \in S(t)\}.
    \end{equation}
    Notice that if a vector $(- v_t)\otimes v$ is minimal in $S_0(t)$, then $v$ is minimal in $S_v(t)$. Now
    \begin{itemize}
        \item $\beta_1$ and $(- a_{11})\otimes\beta_4$ are minimal in
        \begin{equation*}
            S_{0}(1) = \{\begin{bmatrix}
		    0 \\
		    \ninf\\
		    \ninf\\
		    0
	    \end{bmatrix}, \begin{bmatrix}
		    0 \\
		    a_{11} - a_{21}\\
		    \ninf\\
		    0
	    \end{bmatrix}, \begin{bmatrix}
	        0\\
		    \ninf \\
		    a_{11} - a_{12}\\
		    0
	    \end{bmatrix}, \begin{bmatrix}
		    0 \\
	        a_{12} - a_{11}\\
		    a_{21} - a_{11}\\
		    \ninf
	    \end{bmatrix}\}
        \end{equation*}
        \item $(-a_{11})\otimes\beta_2$ and $(-a_{12}) \otimes \beta_6$ are minimal in
        \begin{equation*}
            S_{0}(2) = \{\begin{bmatrix}
		    a_{21} - a_{11}\\
		    0 \\
		    \ninf\\
		    a_{21} - a_{11}
	    \end{bmatrix}, \begin{bmatrix}
		    a_{11} - a_{12}\\
	        0\\
		    a_{21} - a_{12}\\
		    \ninf
	    \end{bmatrix},\begin{bmatrix}
		    \ninf \\
	        0\\
		    a_{21} - a_{12}\\
		    a_{22} - a_{12}
	    \end{bmatrix},\begin{bmatrix}
		    \ninf \\
	        0\\
		    a_{21} - a_{12}\\
		    \ninf
	    \end{bmatrix}\}
        \end{equation*}
        \item $(-a_{11}) \otimes \beta_3$ is minimal in
        \begin{equation*}
            S_{0}(3) = \{ \begin{bmatrix}
	        a_{12} - a_{11}\\
		    \ninf \\
		    0\\
		    a_{12} - a_{11}
	    \end{bmatrix}, \begin{bmatrix}
		    a_{11} - a_{21} \\
	        a_{12} - a_{21}\\
		    0\\
		    \ninf
	    \end{bmatrix},\begin{bmatrix}
		    \ninf \\
	        a_{12} - a_{21}\\
		    0 \\
		    a_{22} - a_{21}
	    \end{bmatrix},\begin{bmatrix}
		    \ninf \\
	        a_{12} - a_{21}\\
		    0\\
		    \ninf
	    \end{bmatrix}\}
        \end{equation*}
        \item $(- a_{22}) \otimes \beta_5$ is minimal in
        \begin{equation*}
            S_{0}(4) = \{\begin{bmatrix}
		    0 \\
		    \ninf\\
		    \ninf\\
		    0
	    \end{bmatrix}, \begin{bmatrix}
		    0 \\
		    a_{11} - a_{12}\\
		    \ninf\\
		    0
	    \end{bmatrix}, \begin{bmatrix}
	        0\\
		    \ninf \\
		    a_{11} - a_{21}\\
		    0
	    \end{bmatrix}, \begin{bmatrix}
	        \ninf \\
	        a_{12} - a_{22}\\
		    a_{21} - a_{22}\\
		    0
	    \end{bmatrix}\}
        \end{equation*}
    \end{itemize}
    Hence, all $\beta_i, i = 1,2,3,4,5,6$ are extremals in $K$ and form an independent set. Then $S$ is a basis of $K$.
\end{proof}

\begin{rem}
    We emphasise that $\beta_2$ and $\beta_3$ in Lemma \ref{lem: a11 = a22} cannot be replaced by 
    \begin{equation*}
        \beta'_2 = \begin{bmatrix}
		    0 \\
		    \alpha_1\\
		    \ninf\\
		    0
	    \end{bmatrix}
    	\quad
    	\beta'_3 = \begin{bmatrix}
		    0 \\
		    \ninf\\
		    \alpha_2\\
		    0
	    \end{bmatrix}
    \end{equation*}
    where $\alpha_1 = \min(a_{12} - a_{11}, a_{11} - a_{21}), \alpha_2 = \min(a_{21} - a_{11}, a_{11} - a_{12})$, as in Lemma \ref{lem: a11 < a22}. Suppose, for example, $(a_{11} \otimes a_{11}) > (a_{12} \otimes a_{21})$ and $\beta_2$ and $\beta_3$ are replaced by $\beta'_2$ and $\beta'_3$ in $S$. Then $a_{12} - a_{11} < a_{11} - a_{12}$, as well as $a_{21} - a_{11} < a_{11} - a_{12}$. Then
    \begin{equation*}
        \beta'_2 = \begin{bmatrix}
		    0 \\
		    a_{12} - a_{11}\\
		    \ninf\\
		    0
	    \end{bmatrix}
    	\quad
    	\beta'_3 = \begin{bmatrix}
		    0 \\
		    \ninf\\
		    a_{21} - a_{11}\\
		    0
	    \end{bmatrix}
    \end{equation*}
    The vector 
    \begin{equation*}
        b = \begin{bmatrix}
		    0\\
		    a_{11} - a_{21}\\
		    a_{21} - a_{11}\\
		    0
	    \end{bmatrix}
    \end{equation*}
    cannot be a linear combination of $S$ but 
    \begin{equation*}
            A \otimes B_{b} = \begin{bmatrix}
        		a_{11} & a_{11} + a_{11} - a_{21}\\
        		a_{21} & a_{11}\\
        	\end{bmatrix}
        	= B_{b} \otimes A
    \end{equation*}
\end{rem}
\section{Geometric Representation of the Commuting Polyhedral Cone}\footnote{A set of R codes which visualises the Barycentric coordinates of the commuting tropical polyhedral cone given a finite $2 \times 2$ matrix $A$ (as example \ref{eg1})is available on the public GitHub repository \url{https://github.com/Xieyangxinyu/On-2-by-2-Tropical-Commuting-Matrices}.}
\label{geo}
Barycentric coordinates give a convenient approach to visualise the four extremals of $\beta$s when $a_{11} > a_{22}$ or $a_{11} < a_{22}$. When $a_{11} > a_{22}$, $b_{11}$ is zero in four different scaled extremals. Hence, we can achieve a nice visualisation in $\Rbar^3$ as a projection of the polyhedral cone onto the three other axes $(b_{12}, b_{21}, b_{22})$. To avoid the burden of negative infinity, we use the exponential transformations of the original entries as the coordinates in the Barycentric triangle after normalisation. That is, for a given set of scaled $\beta$s,
	\begin{equation}
	\label{extreme2}
		\begin{split}
			\beta_1 &= (0, \ninf, \ninf, 0)\\
			\beta_2 &= (0, \alpha_1, \ninf, 0)\\
			\beta_3 &= (0, \ninf, \alpha_2, 0)\\
			\beta_4 &= (0, a_{12} - a_{11}, a_{21} - a_{11}, \ninf)
		\end{split}
	\end{equation}
the corresponding barycentric coordinates are
	\begin{equation}
		\begin{split}
			\beta'_1 &= (0, 0, 1)\\
			\beta'_2 &= (\frac{\exp(\alpha_1)}{\exp(\alpha_1) + 1}, 0, \frac{1}{\exp(\alpha_1) + 1})\\
			\beta'_3 &= (0, \frac{\exp(\alpha_2)}{\exp(\alpha_2) + 1}, \frac{1}{\exp(\alpha_2) + 1})\\
			\beta'_4 &= (\frac{\exp(a_{12} - a_{11})}{\exp(a_{12} - a_{11}) + \exp(a_{21} - a_{11})}, \frac{\exp(a_{21} - a_{11})}{\exp(a_{12} - a_{11}) + \exp(a_{21} - a_{11})}, 0)
		\end{split}
	\end{equation}
One typical representation of the projection of the tropical polyhedral cone generated by $\beta$s in (\ref{extreme2}) is composed of the shaded areas and the thick line segments shown on the left hand side of Figure \ref{proj}. In the rest of this section, we prove that the three line segments $\beta_1\beta_4, b_{21}\beta_2, b_{12}\beta_3$ intersect at the same point $\omega$, as shown on the right hand side of Figure \ref{proj}.

\begin{figure}
\centering
\begin{tikzpicture}[font=\LARGE] 

\def \l{     5.50000000000000 } 

\coordinate (A) at (0,0); 
\coordinate (B) at ({\l}, 0);
\coordinate (C) at ({\l / 2 }, {\l * cos(30)});
\coordinate (D) at ({0.2749170 * \l}, {0.4761702 * \l });
\coordinate (E) at ({0.5249792 *\l}, 0);
\coordinate (F) at ({0.7375104*\l}, {(0.4546453 * \l});
\coordinate (I) at (0.5158077*\l, 0.3179746*\l);

\fill[black]  (A) circle [radius=2pt]; 
\fill[black]    (B) circle [radius=2pt]; 
\fill[black]    (C) circle [radius=2pt]; 
\fill[black]    (D) circle [radius=2pt]; 
\fill[black]    (E) circle [radius=2pt]; 
\fill[black]    (F) circle [radius=2pt]; 

\filldraw[draw=gray,bottom color=gray!40, top color=gray!40]
  (C) -- (D) -- (I) -- (F);
  
\draw[black,-] 
                        (B)  --  (A);
\draw[black,-] 
                        (B)  --  (C);
\draw[black,-] 
                        (C)  --  (A);
\draw[black, dotted, -, thick] 
                        (B)  --  (D);
\draw[black, dotted, -, thick] 
                        (C)  --  (E);
\draw[black, dotted, -, thick] 
                        (A)  --  (F);
\draw[black,-, very thick] 
                        (I)  --  (E);
\draw[black,-, very thick] 
                        (I)  --  (D);
\draw[black,-, very thick] 
                        (I)  --  (F);
\draw[black,-, very thick] 
                        (I)  --  (E);
\draw[black,-, very thick] 
                        (C)  --  (D);
\draw[black,-, very thick] 
                        (C)  --  (F);

\draw 
      (A)  node [left,gray]       {$b_{12}$}
      (B)  node [right,gray]       {$b_{21}$}
      (C)  node [above ,gray]     {$b_{22}$}
      (C)  node [right ,black]     {$\beta_1$}
      (D)  node [above, black]     {$\beta_2$}
      (E)  node [below,black]     {$\beta_4$}
      (F)  node [below,black]     {$\beta_3$};
\end{tikzpicture}
\qquad
\begin{tikzpicture}[font=\LARGE] 
\def \l{     5.50000000000000 } 

\coordinate (A) at (0,0); 
\coordinate (B) at ({\l}, 0);
\coordinate (C) at ({\l / 2 }, {\l * cos(30)});
\coordinate (D) at ({0.2749170 * \l}, {0.4761702 * \l });
\coordinate (E) at ({0.5249792 *\l}, 0);
\coordinate (F) at ({0.7375104*\l}, {(0.4546453 * \l});
\coordinate (I) at (0.5158077*\l, 0.3179746*\l);
\coordinate (A_1) at (0.37 * \l,0.3 * \l); 
\coordinate (A_2) at (0.4 * \l, 0.1 * \l);
\coordinate (A_3) at (0.7 * \l, 0.1 * \l);
\coordinate (A_4) at (0.8 * \l, 0.35 * \l);
\coordinate (A_5) at (0.68 * \l, 0.5 * \l);
\coordinate (A_6) at (0.47 * \l, 0.5 * \l);

\fill[black]    (A) circle [radius=2pt]; 
\fill[black]    (B) circle [radius=2pt]; 
\fill[black]    (C) circle [radius=2pt]; 
\fill[black]    (D) circle [radius=2pt]; 
\fill[black]    (E) circle [radius=2pt]; 
\fill[black]    (F) circle [radius=2pt]; 
\fill[black]    (I) circle [radius=2pt];

\draw[black,-] 
                        (B)  --  (A);
\draw[black,-] 
                        (B)  --  (C);
\draw[black,-] 
                        (C)  --  (A);
\draw[black, dotted, -, thick] 
                        (B)  --  (D);
\draw[black, dotted, -, thick] 
                        (C)  --  (E);
\draw[black, dotted, -, thick] 
                        (A)  --  (F);

\draw 
      (A)  node [left,gray]       {$b_{12}$}
      (B)  node [right,gray]       {$b_{21}$}
      (C)  node [above ,gray]     {$b_{22}$}
      (C)  node [right ,black]     {$\beta_1$}
      (D)  node [above, black]     {$\beta_2$}
      (E)  node [below,black]     {$\omega'$}
      (F)  node [above,black]     {$\beta_3$}
      (I)  node [left,black]     {$\omega$}
      (A_1)  node [left,black]       {$A_1$}
      (A_2)  node [left,black]       {$A_2$}
      (A_3)  node [left,black]       {$A_3$}
      (A_4)  node [left,black]       {$A_4$}
      (A_5)  node [left,black]       {$A_5$}
      (A_6)  node [left,black]       {$A_6$};
\end{tikzpicture}
\caption{The projection of a tropical polyhedral cone}
\label{proj}
\end{figure}
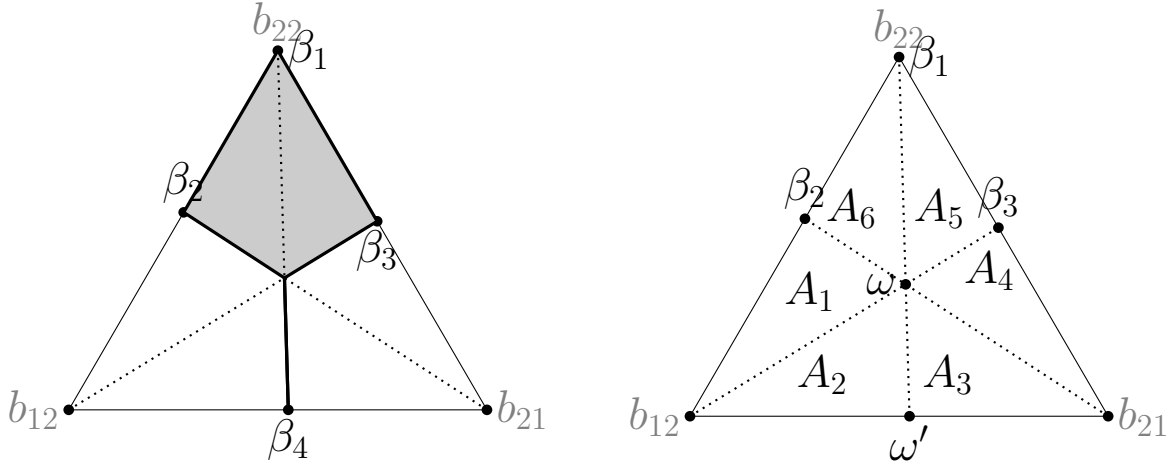

\begin{lem}
	\label{hint}
	$\alpha_2 = \alpha_1 \otimes (a_{21} - a_{12})$.
\end{lem}
\begin{proof}
	By adding $(a_{21} - a_{12})$ to $\alpha_1$, we obtain $\alpha_1 \otimes (a_{21} - a_{12}) = \min(a_{21} - a_{11}, a_{22} - a_{12}) = \alpha_2$.
\end{proof}

\begin{figure}
\centering
\begin{tikzpicture}[font=\LARGE] 
\def \l{     5.00000000000000 } 

\coordinate (A) at (0,0); 
\coordinate (B) at ({\l}, 0);
\coordinate (C) at ({\l / 2 }, {\l * cos(30)});
\coordinate (D) at ({0.2674715 * \l}, {0.4632742 * \l });
\coordinate (E) at ({0.1852764 *\l}, 0);
\coordinate (F) at ({0.5825335 *\l}, {(0.7230732 * \l});
\coordinate (I) at (0.3375331*\l, 0.4189650*\l);

\fill[black]  (A) circle [radius=2pt]; 
\fill[black]    (B) circle [radius=2pt]; 
\fill[black]    (C) circle [radius=2pt]; 
\fill[black]    (D) circle [radius=2pt]; 
\fill[black]    (E) circle [radius=2pt]; 
\fill[black]    (F) circle [radius=2pt]; 

\filldraw[draw=gray,bottom color=gray!40, top color=gray!40]
  (C) -- (D) -- (I) -- (F);
  
\draw[black,-] 
                        (B)  --  (A);
\draw[black,-] 
                        (B)  --  (C);
\draw[black,-] 
                        (C)  --  (A);
\draw[black, dotted, -, thick] 
                        (B)  --  (D);
\draw[black, dotted, -, thick] 
                        (C)  --  (E);
\draw[black, dotted, -, thick] 
                        (A)  --  (F);
\draw[black,-, very thick] 
                        (I)  --  (E);
\draw[black,-, very thick] 
                        (I)  --  (D);
\draw[black,-, very thick] 
                        (I)  --  (F);
\draw[black,-, very thick] 
                        (I)  --  (E);
\draw[black,-, very thick] 
                        (C)  --  (D);
\draw[black,-, very thick] 
                        (C)  --  (F);

\draw 
      (A)  node [left,gray]       {$b_{12}$}
      (B)  node [right,gray]       {$b_{21}$}
      (C)  node [above ,gray]     {$b_{22}$}
      (C)  node [right ,black]     {$\beta_1$}
      (D)  node [above, black]     {$\beta_2$}
      (E)  node [below,black]     {$\beta_4$}
      (F)  node [below,black]     {$\beta_3$};
\end{tikzpicture}
\label{exam1}
\caption{The tropical polyhedral cone projection for example \ref{eg1}}
\end{figure}
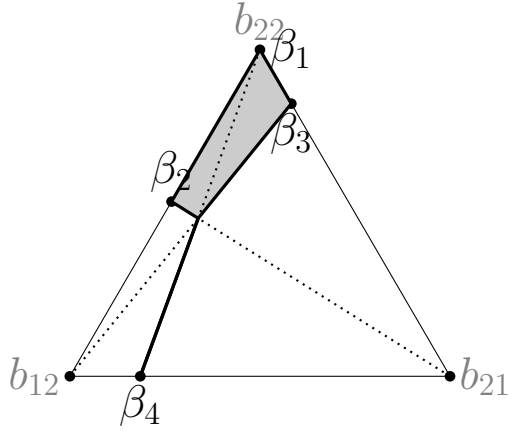

\begin{lem}
	\label{coinc}
	Suppose the two line segments $\beta_2b_{21}$ and $\beta_3b_{12}$ intersect at the point $\omega$. We draw a line from $b_{22}$ to $\omega$ and extend it until it intersects with $b_{12}b_{21}$ at point $\omega'$. Then, $\omega'$ coincides with $\beta_4$.
\end{lem}
\begin{proof}
	The three line segments in the lemma partition the triangle $T$ into 6 subareas, as shown on the right hand side of Figure \ref{proj}. Using the definition of Barycentric coordinates in (\ref{BaryCoor}), the barycentric coordinate $\beta'_2$ shows that 
	\begin{equation*}
		\frac{A_4 + A_5}{A_2 + A_3} = \frac{A_4 + A_5 + A_6}{A_1 + A_2 + A_3} = \frac{\frac{\exp(\alpha_1)}{\exp(\alpha_1) + 1}}{\frac{1}{\exp(\alpha_1) + 1}} = \frac{\exp(\alpha_1)}{1}
	\end{equation*}
Similarly, barycentric coordinate $\beta'_3$ shows that 
	\begin{equation*}
		\frac{A_1 + A_6}{A_2 + A_3} = \frac{A_1 + A_5 + A_6}{A_2 + A_3 + A_4} = \frac{\frac{\exp(\alpha_2)}{\exp(\alpha_2) + 1}}{\frac{1}{\exp(\alpha_2) + 1}} = \frac{\exp(\alpha_2)}{1}
	\end{equation*}
Hence, the information of $\omega'$ can be obtained by
	\begin{equation*}
		\frac{A_3 + A_4 + A_5}{A_1 + A_2 + A_6} = \frac{A_4 + A_5}{A_1 + A_6} = \frac{\exp(\alpha_1)}{\exp(\alpha_2)}
	\end{equation*}
By Lemma \ref{hint}, we have 
\begin{equation*}
	\frac{\exp(\alpha_1)}{\exp(\alpha_2)} = \frac{1}{\exp(a_{21} - a_{12})} = \frac{\exp(a_{12} - a_{11})}{\exp(a_{21} - a_{11})}
\end{equation*}
Hence, the coordinates of $\omega'$ and $\beta_4$ coincide.
\end{proof}

\begin{thm}
	Given a set of scaled extremals as (\ref{extreme2}), the three line segments $\beta_1\beta_4, b_{21}\beta_2, b_{12}\beta_3$ intersect at the same point.
\end{thm}
\begin{proof}
	This follows readily from Lemma \ref{coinc}.
\end{proof}
\begin{exmp}
	\label{eg1}
	Given a finite matrix 
	\begin{equation*}
		A = \begin{bmatrix}
		0.166 & 0.861\\
		-0.62 & -0.76\\
	\end{bmatrix},
	\end{equation*}
	the scaled extremals of the tropical polyhedral cones are
	\begin{equation*}
		\begin{split}
			\beta_1 &= (0, \ninf, \ninf, 0)\\
			\beta_2 &= (0, -0.14, \ninf, 0)\\
			\beta_3 &= (0, \ninf, -1.621, 0)\\
			\beta_4 &= (0, 0.695, -0.786, \ninf)
		\end{split}
	\end{equation*}
	Hence, the coordinates of each point in the Barycentric triangle are thus
	\begin{equation*}
		\begin{split}
			\beta'_1 &= (0, 0, 1)\\
			\beta'_2 &= (\frac{\exp(-0.14)}{\exp(-0.14) + 1}, 0, \frac{1}{\exp(-0.14) + 1})\\
			\beta'_3 &= (0, \frac{\exp(-1.621)}{\exp(-1.621) + 1}, \frac{1}{\exp(-1.621) + 1})\\
			\beta'_4 &= (\frac{\exp(0.695)}{\exp(0.695) + \exp(-0.786)}, \frac{\exp(-0.786)}{\exp(0.695) + \exp(-0.786)}, 0)
		\end{split}
	\end{equation*}
\end{exmp}

\section*{Acknowledgement}
Many thanks to Dr Ngoc M. Tran for guidance throughout this project. Also thanks to George D. Torres and Luyan Yu for very helpful insights in tropical polyhedral cones and  Yaoyang Liu for comments on earlier drafts.

\bibliography{mybibfile}

\end{document}